\documentclass{amsart}
\usepackage[parfill]{parskip} 
\usepackage{amsmath,amssymb,amsthm}
\usepackage[pdftex]{hyperref}
\usepackage{euscript}
\usepackage{tikz-cd}
\usepackage{appendix}
\usepackage{comment}
\usepackage{mathtools}
\usepackage{graphicx} 
\usepackage{booktabs}

\hypersetup{
  colorlinks,
  citecolor=black,
  filecolor=black,
  linkcolor=black,
  urlcolor=black
}

\numberwithin{equation}{section}

\theoremstyle{plain}
\newtheorem{theorem}{Theorem}[section]
\newtheorem*{theorem*}{Theorem}
\newtheorem{lemma}[theorem]{Lemma}
\newtheorem{proposition}[theorem]{Proposition}
\newtheorem{corollary}[theorem]{Corollary}

\theoremstyle{definition}
\newtheorem{definition}[theorem]{Definition}

\newtheorem{remark}[theorem]{Remark}

\DeclareMathOperator{\Eq}{Equiv}

\DeclareMathOperator{\op}{op}

\begin{document}
\title{Rings and Boolean Algebras as Algebraic Theories}

%
\author{Arturo De Faveri} 
%
%
\email{defaveri@irif.fr}
\address{IRIF\\
CNRS and Universit\'e Paris Cit\'e\\
8 Place Aur\'elie Nemours, 75013 Paris, France
}

\begin{abstract}
We present a unified framework for representing commutative rings through affine algebraic theories and 
Boolean rings through hyperaffine algebraic theories.
This yields categorical equivalences between these theories and, respectively, commutative rings and Boolean rings. 
We then analyse models of affine theories over a Boolean ring $B$, comparing them with the models of hyperaffine theories, the well-known $B$-sets. 
Two novel characterisations are presented: the first defines these models as Boolean vector spaces equipped with an action of the Boolean ring; 
the second provides a representation in terms of sheaves, in analogy with $B$-sets.
Finally, we establish a connection between hyperaffine theories and multidimensional Boolean algebras, a recently introduced generalisation of Boolean algebras.   

\end{abstract}

\maketitle

\section{Introduction}
An \emph{algebraic theory} (also called a \emph{clone}) is an abstraction of the notion of algebraic operations and the equations they satisfy.
In universal algebra, an equational class of algebras is specified by a signature of operations and a set of equational identities.
An algebraic theory $T$ encodes the same information in different terms: it is a collection of abstract operations of finite arities together with a rule for their composition.  
An \emph{algebra} of the theory $T$ is then a set equipped with actual operations corresponding to the elements of the theory, in such a way that composition is respected.

Given an algebraic structure, a natural and intriguing question is how to represent it as an algebraic theory.
In this paper we focus on rings, and we ask whether there exists a full embedding $F: \mathbf{Ring} \to \mathbf{AlgTh}$ of the category of rings into that of algebraic theories such that suitable abstract properties of the theory $F(R)$ enable to reconstruct perfectly the ring $R$. 
The first idea that comes to mind is to associate with $R$ the theory $F(R)$ of (left) $R$-modules. 
The abstract operations, in this case, are the tuples $(r_1,\ldots,r_n) \in R^n$, thought of as linear combinations $r_1x_1+\cdots+r_nx_n$. 
If the ring is commutative, then all the abstract operations are commutative.
Two binary operations $s,r$ are said to commute if 
\begin{equation}
  \label{eq:comm}
r(s(x,y),s(z,w)) = s(r(x,z),r(y,w)) \text{,}
\end{equation}
and this condition can be generalised to higher arities. 
However, if the goal is to reconstruct $R$ from abstract properties like \eqref{eq:comm}, the encoding of $R$ via $R$-modules is not sufficiently rich. 

It turns out that there are two orthogonal constructions that allow us to reconstruct $R$ perfectly
and these two constructions have in common the idea of encoding the ring through the binary operations of the corresponding theory. 
The first, which works in the case of a commutative ring, associates with $R$ the theory of affine $R$-modules, namely the theory whose operations are linear combinations $r_1x_1+\cdots+r_nx_n$ such that $r_1+\cdots+r_n=1$. 
The operations of this theory are commutative and idempotent.
The second construction, specific to Boolean rings, associates with a Boolean ring $B$ the theory of ``hyperaffine'' $B$-modules, that is, the theory whose operations are linear combinations $b_1x_1+\cdots+b_nx_n$ such that $b_1+\cdots+b_n=1$ and $b_i b_j =0$ for $i \neq j$. 
In this case, a strong property holds
\begin{equation}
b(b(x,y), b(z,w)) = b(x,w) \text{,}
\end{equation}
making each binary operation a rectangular band. 
These observations can be crystallised in the following statement:
for every commutative (resp. Boolean) ring $R$, the theory of affine (resp. ``hyperaffine'') $R$-modules is affine (resp. hyperaffine); 
moreover, any affine (resp. hyperaffine) theory is uniquely determined by the ring of its binary operations 
(cf. Theorems \ref{thm:equibool} and \ref{thm:equiring}).

Our first main contribution is the development of a convenient framework in which the ``hyperaffine'' and ``affine'' constructions can be realised.
We first 
provide a new proof that hyperaffine theories form a category equivalent to that of Boolean rings.
This result was originally stated in \cite[Theorem 4.9]{G24}.
We then apply the same strategy to show in Theorem~\ref{thm:equiring} that affine theories correspond precisely to commutative rings.
We highlight the importance of the role played by the \emph{coefficients} of a theory to the end of presenting both constructions in a single, coherent setting.
Given a Boolean ring $B$, as $B$ is commutative, we have two ways of encoding $B$ as an algebraic theory: either as a hyperaffine theory or as an affine theory. 
Models of hyperaffine theories are known to be $B$-sets \cite{B91}, so that we are led to the question: what are the models of an affine theory whose ring of coefficients is Boolean? 
Our second main contribution is the answer to this question.
We provide two characterisations for the models of an affine theory over a Boolean ring $B$.
The first, Theorem \ref{thm:new1}, is purely algebraic:
the models of an affine theory over a Boolean ring are Boolean vector spaces equipped with a compatible action of $B$ which is idempotent and commutative. 
The second, Theorem \ref{thm:new2}, shows that these structures admit a sheaf representation analogous to $B$-sets (see \cite{B91}).

Inspired by the representation of Boolean algebras through a ternary ``conditional disjunction" operation (if-then-else), some authors have recently introduced $n$-dimensional Boolean algebras ($n$BAs), featuring an $(n+1)$-ary operation $q$ and $n$ constants $e_1, \ldots, e_n$ \cite{BS,BLPS18,SBLP20}. 
These structures provide the algebraic semantics for a many-valued logic with symmetric truth values $e_1, \ldots, e_n$. 
We highlight a connection between hyperaffine algebraic theories and $n$BAs. 
We prove in Theorem \ref{thm:nba} that every hyperaffine theory $T$ naturally induces an $n$BA structure on the set of its $n$-ary operations; 
conversely, a coherent sequence of $n$BAs uniquely determines a hyperaffine theory. 
This link originates from the dual interpretation of the $(n+1)$-ary operator $q$, viewed on the one hand as a composition operator, and on the other as a generalised if-then-else.
In \cite{BS}, it is shown that $n$BAs are related to skew Boolean algebras \cite{Leech1}.  
These considerations suggest that, in the affine case, the corresponding theories naturally point toward the introduction of suitable notions of $n$-dimensional commutative rings and skew rings, in analogy with the Boolean case.

\subsection{Related works.}
Hyperaffine theories were introduced by Johnstone \cite{J90} in the context of a theorem providing necessary and sufficient syntactic conditions for a variety of algebras to form a cartesian closed category.
This result has since been elegantly revisited, opening up new perspectives and developments \cite{G24,G25}.
Among other contributions, \cite{G24} offers a detailed account of the connection between hyperaffine theories and their models, namely $B$-sets.
More about $B$-sets can be found in \cite{S98a,S98b}.
On the other hand, models of affine theories essentially capture what the Polish school calls Mal'cev modes \cite{RS02},
culminating a long line of research on varieties of affine modules \cite{SO66,C75,S77,PRS95}.
The models of hyperaffine theories are connected with the algebraic treatment of the if-then-else construct, a topic that has been extensively studied in the semantics of programming languages \cite{M90,M92,M93,JS09}. 

\subsection{Outline of the paper.}
Section~\ref{sec:alg} introduces the preliminary material on algebraic theories.
Section~\ref{sec:repr} is divided into two main parts.
In the first, it is shown that hyperaffine theories form a category equivalent to that of Boolean rings.
In the second part, we apply the previously developed blueprint to show that affine theories, by contrast, correspond to commutative rings.
Section \ref{sec:nba} explores the link between hyperaffine theories and $n$-dimensional Boolean algebras.
Section~\ref{sec:spaces} is devoted to the study of the models of these theories, with particular attention to the models of affine theories when the ring of coefficients is Boolean.
Finally, Section \ref{sec:conclusions} wraps up with concluding remarks and directions for future work.

\section{Varieties and algebraic theories}
\label{sec:alg}
In a broad sense, a variety is the category of models of a finitary and one-sorted equational theory. 
A given variety may be axiomatised by operations and equations in many ways; 
however, there is always a canonical choice, which is captured by the notion of \emph{algebraic theory} due to Lawvere \cite{Law63} (see also \cite[Section 3]{B2} and \cite[Chapter 11]{ARV10}). 
Here we give a formulation of algebraic theory essentially equivalent to what universal algebraists call an abstract clone \cite{T93}.

\begin{definition}
  An \emph{algebraic theory} $T$ is a family of sets $T(n)$, for each $n \in \mathbb{N}$, together with elements $\pi^n_i \in T(n)$, $1 \le i \le n$, and 
a composition operator $ \circ : T(n) \times T(m)^n \to T(m)$, for each $m,n \in \mathbb{N}$. 
These data have to satisfy, writing $f(g_1, \ldots, g_n)$ in place of $f \circ (g_1, \ldots, g_n)$, two unit laws for $1 \le i \le n$
\begin{equation}
  \label{eq:c12}
  \pi^n_i(g_1, \ldots, g_n) = g_i \quad \text{ and } \quad  f(\pi^n_1, \ldots, \pi^n_n) = f 
\end{equation}
and the associative law 
\begin{equation}
  \label{eq:c3}
  f(g_1(h_1,\ldots h_m),\ldots,g_n(h_1,\ldots,h_m)) = f(g_1,\ldots,g_n)\circ (h_1,\ldots h_m)\text{.}
\end{equation}
\end{definition}

We call an element of $T(n)$ an $n$-ary \emph{operation} in $T$. 

\begin{definition}
  A morphism of theories $\varphi: T \to T'$ comprises, for every $n$, a function $\varphi: T(n) \to T'(n)$ satisfying
\begin{equation}
  \varphi(\pi^n_i) = \pi^n_i \text{ for } 1 \le i \le n  \text{ and } \varphi(f(g_1, \ldots, g_n)) = \varphi(f)(\varphi(g_1), \ldots, \varphi(g_n))
\end{equation}
for all $f \in T(n)$ and $g_1, \ldots, g_n \in T(m)$. 
\end{definition}


\begin{definition}
A \emph{model} or an \emph{algebra} for an algebraic theory $T$ is a set $X$ 
together with, for each $n \in \mathbb{N}$, a left action $\cdot : T(n) \times X^n \to X$ satisfying two conditions: 
  \begin{equation}
    \label{eq:alpha1}
    \pi^n_i \cdot (x_1,\ldots,x_n) = x_i 
  \end{equation}
for all $1 \le i \le n$, and 
  \begin{equation}
    \label{eq:alpha2}
    f(g_1, \ldots, g_n) \cdot (x_1, \ldots, x_m) = f \cdot (g_1 \cdot (x_1, \ldots, x_m), \ldots, g_n \cdot (x_1, \ldots, x_m))
  \end{equation}
  for each $f \in T(n), g_1, \ldots, g_n \in T(m), x_1, \ldots, x_m \in X$.
  The category of models of a theory is called an \emph{algebraic variety}. 
\end{definition}

For each $m$, the set $T(n)$ is a model with action 
\begin{equation*}
  \cdot: T(n) \times T(m)^n \to T(n) \quad \text{ given by } \quad \circ : T(n) \times T(m)^n \to T(n)\text{.} 
\end{equation*}
We call $T(m)$ the free model on $n$ generators.
There is an obvious forgetful functor from the category of models of $T$ to the category of sets, associating with the model $X$ its underlying set $X$.
The underlying set of the free model on $n$ generators is $T(n)$. 
In this light, we can freely identify the operation $f \in T(n)$ with the element $f(x_1, \ldots, x_n)$ in the free model on the $n$ generators $x_1, \ldots, x_n$.

The initial theory $S$ is such that $S(n)$ contains only the $n$ projections $\pi^n_i$ for $n \ge 1$ and $S(0)$ is empty; models of $S$ are mere sets. 
The terminal theory $U$ is given by $U(n)=\{*\}$ for all $n$ and its only model is the singleton. 
The terminal theory has a subtheory $U'$ (i.e., there is an injective morphism $U' \to U$) given by 
\begin{equation*}
    U'(n) = 
    \begin{cases}
      \varnothing & \text{ if } n = 0 \\
      \{*\} & \text{ if } n \neq 0 \text{,}
    \end{cases}
  \end{equation*}
whose models are either the empty set or the singleton. 
We refer to $U$ and to $U'$ as the two \emph{degenerate} theories. 

\subsection{Mal'cev operations, affine and hyperaffine theories.}
We recall the following well-known definitions. 

\begin{definition} 
  \label{def:icd}
  Let $T$ be an algebraic theory; $f \in T(n)$ is said to 
  \begin{itemize}
  \item be \emph{idempotent} if $f(x,\ldots,x)=x$; 
  \item \emph{split} if 
    $f(f(x^1_1, \ldots x^1_n), \ldots, f(x^n_1, \ldots, x^n_n)) = f(x^1_1, \ldots, x^n_n)\text{.}$
  \end{itemize}
  An idempotent, splitting operation is called a \emph{decomposition operation}. 
  Moreover, $f \in T(n)$ and $g \in T(m)$
  are said to \emph{commute} if 
    \begin{equation*}
      f(g(x^1_1, \ldots x^1_m), \ldots, g(x^n_1, \ldots x^n_m)) = g(f(x^1_1, \ldots x^n_1), \ldots, f(x^1_m, \ldots x^n_m))\text{.}
    \end{equation*}
  A theory is \emph{idempotent} (resp. \emph{commutative}) if every operation is (resp. if every pair of operations commute). 
\end{definition}

Another essential ingredient in the paper is that of a Mal'cev operation. 
Roughly speaking, a Mal'cev operation in a commutative theory generalises the operation $f(x,y,z)=x-y+z$ in Abelian groups. 
The importance of this concept lies in the fact that the presence of a Mal'cev operation in a theory is reflected in the internal structure of its models: 
given two congruences $R$, $S$ on a model $X$ of $T$, their composition $R \circ S$ commutes, meaning that $R \circ S = S \circ R$, iff $T(3)$ has a Mal'cev operation.
Formally, a Mal'cev operation is defined as follows; 
for an exhaustive and much more general treatment we refer the reader to \cite{Smith,BB04}. 

\begin{definition} 
  Let $T$ be an algebraic theory. 
  An element $p \in T(3)$ is called a \emph{Mal'cev operation} if 
  \begin{equation*}
    x=p(x,y,y) \quad \text{ and } \quad p(x,x,y)=y \text{.}
  \end{equation*}
\end{definition}


\begin{lemma}
  \label{lem:asscomm}
  Let $p \in T(3)$ be a Mal'cev operation that commutes with itself. 
  Then $p$ satisfies 
    \begin{description}
    \item[M1] $p(x,y,p(t,u,v)) = p(p(x,y,t),u,v)$; 
    \item[M2] $p(x,y,z) = p(z,y,x)$. 
  \end{description}
  Moreover, if $p$ and $p'$ are two Mal'cev operations in $T(3)$ which commute with each other, then they coincide.  
\end{lemma}
\begin{proof}
  Firstly, we prove (M1)  
  \begin{align*}
    p(x,y,p(t,u,v)) & = p(p(x,z,z),p(y,z,z),p(t,u,v)) \\
    & = p(p(x,y,t),p(z,z,u),p(z,z,v)) \\
    & = p(p(x,y,t),u,v)\text{.}
  \end{align*}
  Then we prove (M2)
  \begin{align*}
    p(x,y,z) & = p(p(z,y,x),p(z,y,x),p(x,y,z)) & \\
    & = p(z,y,p(x,p(z,y,x),p(x,y,z))) & \text{(M1)} \\
    & = p(z,y,p(p(z,z,x),p(z,y,x),p(x,y,z))) & \\
    & = p(z,y,p(p(z,z,x),p(z,y,y),p(x,x,z))) &\\
    & = p(z,y,p(x,z,z))& \\
    & = p(z,y,x)\text{.}& 
  \end{align*} 
  Finally, let $p$ and $p'$ be two Mal'cev operations that commute with each other.
  Then 
  \begin{align*}
    p(x,y,z) & = p(p'(x,y,y),p'(y,y,y),p'(y,y,z)) \\
    & = p'(p(x,y,y),p(y,y,y),p(y,y,z)) \\
    & = p'(x,y,z)
  \end{align*}
  concluding the proof. 
\end{proof}

\begin{definition}
  \label{def:hyper-affine}
  We call \emph{affine}\footnote{Note that this terminology is slightly non-standard and some authors use the word `affine' to denote what we call idempotent theories.} an idempotent, commutative theory $T$ with a Mal'cev operation $p \in T(3)$.  
  Moreover, following \cite[Section 4]{J90} we call \emph{hyperaffine} an idempotent, commutative theory in which every operation splits.
\end{definition}

The two degenerate theories $U$ and $U'$ are affine and hyperaffine.

\begin{remark}
  An algebra whose operations are idempotent and commutative is called a \emph{mode} \cite{RS02}. 
  Therefore, models of an affine theory form an algebraic variety of modes. 
  The same applies to a hyperaffine theory. 
\end{remark}

The following distributive property will be useful later. 
\begin{lemma}
  \label{lem:c5}
  Let $T$ be a hyperaffine theory. 
  Then for $f \in T(k)$ and $g_1, \ldots, g_k \in T(n)$,
  \begin{equation}
    \label{eq:c5}
    \begin{split}
        f(g_1, \ldots, g_k) & \circ (f(x^1_1, \ldots, x^k_1), \ldots, f(x^1_n, \ldots, x^k_n)) \\
        & =     f(g_1(x^1_1, \ldots, x^1_n), \ldots, g_k(x^k_1, \ldots, x^k_n))\text{.}
    \end{split}
  \end{equation}
  Moreover, if an idempotent theory $T$ satisfies \eqref{eq:c5}, then $T$ is hyperaffine. 
\end{lemma}
\begin{proof}
    We prove \eqref{eq:c5} 
    \begin{align*}
        & f(g_1, \ldots, g_k) \circ (f(x^1_1, \ldots, x^k_1), \ldots, f(x^1_n, \ldots, x^k_n)) \\
        & = f(g_1(f(x^1_1, \ldots, x^k_1), \ldots, f(x^1_n, \ldots, x^k_n)), \ldots, g_k(f(x^1_1, \ldots, x^k_1), \ldots, f(x^1_n, \ldots, x^k_n)))\\
        & = f(f(g_1(x^1_1, \ldots, x^1_n), \ldots, g_1(x^k_1, \ldots, x^k_n)), \ldots, f(g_k(x^1_1, \ldots, x^1_n), \ldots, g_k(x^k_1, \ldots, x^k_n)))\\
        & = f(g_1(x^1_1, \ldots, x^1_n), \ldots, g_k(x^k_1, \ldots, x^k_n))\text{.}
    \end{align*}
  Now, assume that $T$ satisfies \eqref{eq:c5}. 
  Firstly, every $f \in T(n)$ splits 
  \begin{align*}
    & f(f(x^1_1, \ldots, x^n_1), \ldots, f(x^1_n, \ldots, x^n_n)) &  \\
    & = f(\pi^n_1, \ldots, \pi^n_n) \circ (f(x^1_1, \ldots, x^n_1), \ldots, f(x^1_n, \ldots, x^n_n)) & \\
    & = f(\pi^n_1(x^1_1, \ldots, x^n_1), \ldots, \pi^n_n(x^1_n, \ldots, x^n_n)) & \eqref{eq:c5} \\
    & = f(x^1_1, \ldots, x^n_n)   \text{.} &
  \end{align*}
  Finally, every $f \in T(n)$ and $g \in T(m)$ commute
  \begin{align*}
    & f(g(x^1_1, \ldots x^1_m), \ldots, g(x^n_1, \ldots x^n_m)) & \\
    & = f(g, \ldots, g) \circ (f(x^1_1, \ldots x^n_1), \ldots, f(x^1_m, \ldots x^n_m)) & \eqref{eq:c5}\\
    & = g(f(x^1_1, \ldots x^n_1), \ldots, f(x^1_m, \ldots x^n_m)) & 
  \end{align*}
  proving the claim. 
\end{proof}

\section{Representing rings as algebraic theories}
\label{sec:repr}
A \emph{ring} is a tuple $(R,+,0,\cdot,1)$ such that $(R,+,0)$ is an Abelian group, $(R, \cdot, 1)$ is a multiplicative monoid and multiplication distributes over addition: 
$r(a + b) = ra + rb$ and $(a+b)r=ar+br$ for all $a, b, r \in R$.
The ring $R$ is said to be \emph{commutative} if the monoid is, and 
\emph{Boolean} if $r^2 = r$ for every $r \in R$. 
A Boolean ring is necessarily commutative. 

Any ring $R$ gives rise to an algebraic theory $T_R$ as follows. 
For $n \in \mathbb{N}$, let $T_R(n)$ be the set of tuples $(r_1, \ldots, r_n) \in R^n$.
These are the abstract $n$-ary operations of the theory.
The projections in $T_R(n)$ are given by the standard basis of $R^n$: $(1,0, \ldots, 0), \ldots, (0, \ldots, 0,1)$, 
while composition $\circ : T_R(n) \times T_R(m)^n \to T_R(m)$ is given by matrix multiplication:
\begin{equation*} 
\begin{bmatrix}
      r_1 \\
      \vdots \\
      r_n \\
      \end{bmatrix} \circ  
      \begin{bmatrix}
      a^1_1 & \cdots & a^n_1 \\
      \vdots & \ddots & \vdots \\
      a^1_m & \cdots & a^n_m
      \end{bmatrix} =
      \begin{bmatrix}
      a^1_1r_1 + \cdots + a^n_1 r_n\\
      \vdots \\
      a^1_mr_1 + \cdots + a^n_m r_n 
      \end{bmatrix}
\end{equation*}
Models of the theory $T_R$ are precisely (left) modules over the ring $R$.
\subsection{Hyperaffine theories} 
We start by examining the case of hyperaffine theories. 

\begin{definition}
  For a given Boolean ring $B$, we define the algebraic theory $H_B$ whose operations in $H_B(n)$ are $(b_1, \ldots, b_n) \in B^n$ such that $b_1 + \cdots +b_n =1$ and $b_i b_j =0$ for $i \neq j$. 
  Projections are given by the standard basis of $B^n$ and composition is defined by matrix multiplication as above. 
\end{definition}

Boolean rings can be equivalently described as Boolean algebras defining $a \land b := ab$, $a\lor b= a + (1-a)b$, $\lnot a:=1-a$. 
We will freely switch between the two descriptions, because some results are better expressed in terms of Boolean algebras.

Note that, 
if $a b = a \land b = 0$, then $a +b = a \lor b$. 
As a consequence, the operations in $H_B(n)$ are $(b_1, \ldots, b_n) \in B^n$ such that $b_1 \lor \cdots \lor b_n =1$ and $b_i  \land b_j =0$ for $i \neq j$. 

Observe that $H_B(0)=\varnothing$ unless $B$ is degenerate, that is $0 = 1$, in which case $H_B(0)=\{*\}$.

\begin{lemma}
  \label{lem:bool-hyper}
  Let $B$ be a Boolean ring. 
  The algebraic theory $H_B$ is hyperaffine. 
\end{lemma}
  \begin{proof}
    As $B$ is commutative, $H_B$ is commutative.
    Idempotence is ensured by the condition $b_1 + \cdots +b_n =1$: for every $(b_1, \ldots, b_n) \in H_B(n)$ and $(a_1, \ldots, a_m)\in H_B(m)$
    \begin{equation*} 
    \begin{bmatrix}
      b_1 \\
      \vdots \\
      b_n 
      \end{bmatrix} \circ  
      \begin{bmatrix}
      a_1 & \cdots & a_1 \\
      \vdots & \ddots & \vdots \\
      a_m & \cdots & a_m
      \end{bmatrix} =
      \begin{bmatrix}
      a_1(b_1 + \cdots + b_n)\\
      \vdots \\
      a_m(b_1 + \cdots + b_n) 
      \end{bmatrix}= 
      \begin{bmatrix}
      a_1 \\
      \vdots \\
      a_m \\
      \end{bmatrix}
\end{equation*} 
    The fact that $b_i ^2= b_i$ for all $i$ and $b_i b_j =0$ for $i \neq j$ precisely yields that every operation splits: 
    \small
    \begin{equation*}
    \begin{bmatrix}
      b_1 \\
      \vdots \\
      b_n \\
      \end{bmatrix} \circ  
      \begin{bmatrix}
      a^1_{11}b_1 + \cdots + a^1_{n1}b_n & \cdots & a^n_{n1}b_1 + \cdots + a^n_{n1}b_n \\
      \vdots & \ddots & \vdots \\
      a^1_{1m}b_1 + \cdots + a^1_{nm}b_n & \cdots & a^n_{1m}b_1 + \cdots + a^n_{nm}b_n
      \end{bmatrix} =
      \begin{bmatrix}
      a^1_{11}b_1 + \cdots + a^n_{n1} b_n\\
      \vdots \\
      a^1_{1m}b_1 + \cdots + a^n_{nm} b_n 
      \end{bmatrix}
    \end{equation*}
    \normalsize
    for all $(b_1, \ldots, b_n) \in H_B(n)$.
\end{proof}

Recall that we identify $f \in T(n)$ with the element $f(x_1, \ldots, x_n)$ in the free model on $n$ generators. 

\begin{lemma} 
  \cite[Proposition 4.4]{G24}
  \label{lem:hyper-bool}
    Let $T$ be a hyperaffine theory.  
    Then the set $T(2)$ is a Boolean algebra with the operations defined as follows, for every $a(x,y), b(x,y) \in T(2)$:
    \begin{itemize}
        \item $(a \land b)(x,y):=a(b(x,y),y)$; 
        \item $(a \lor b)(x,y):=a(x,b(x,y))$; 
        \item $(\lnot a)(x,y):=a(y,x)$; 
        \item $0(x,y):=y$ and $1(x,y):=x$. 
    \end{itemize}
\end{lemma}

When the Boolean ring $B$ is degenerate, we have that $H_B=U$, the degenerate theory. 

Johnstone \cite[p.~461]{J90} introduced the notion of coefficient to prove that, if $H_n$ is the hyperaffine theory generated by a single operation of arity $n$, then the category of models of $H_n$ is equivalent to the product $\mathbf{Set}^n$.  
The notion of coefficient will play a central role in what follows. 

\begin{definition}
  \label{def:coord}
Let $T$ be an algebraic theory, and let $f \in T(n)$. 
We call the following elements of $T(2)$ \emph{coefficients} of $f$: 
\begin{equation*}
  f[1](x,y):=f(x,y,\ldots,y), \ldots, f[n](x,y):=f(y,\ldots,y,x)\text{.}
\end{equation*}
\end{definition}

The next proposition shows in particular the fact that any operation in a hyperaffine theory 
is completely determined by 
its coefficients.   

\begin{proposition}
  \label{prop:coord2}
  Let $T$ be a hyperaffine theory and let $B$ be the Boolean ring $T(2)$.  
  The following hold: 
  \begin{enumerate}
    \item if $f \in T(n)$, then $f[1] \lor \cdots \lor f[n] =1$ and $f[i]f[j]=0$ for $i \neq j$; 
    \item for $f,g \in T(n)$, if $f[i]=g[i]$ for all $1 \le i \le n$, then $f=g$; 
    \item for every $(b_1, \ldots, b_n)$ such that $b_1 \lor \cdots \lor b_n =1$ and $b_i b_j =0$ for $i \neq j$, there is $f \in T(n)$ such that $f[i]=b_i$ for all $1 \le i \le n$. 
  \end{enumerate}
\end{proposition}
\begin{proof}
  Regarding the first item, the fact that $f[i] f[j]=0$ for each $i \neq j$ follows applying idempotence and \eqref{eq:c5} of Lemma \ref{lem:c5}; 
    that $f[1] \lor \cdots \lor f[n]=1$
    follows from associating the disjunction to the right and repeatedly applying \eqref{eq:c5}.  
  The second item is \cite[Lemma 4.6(i)]{G24}. 
  Finally, we prove the third item. 
  The proof is by induction on $n$ (cf. the proof of \cite[Proposition 4.7]{G24}).  
  If $n=1$, then $f(x)=x$. 
  Assume that the result holds for $n-1$, and let $(b_1, \ldots, b_n)$ be such that $b_1 \lor \cdots \lor b_n =1$ and $b_i \land b_j =0$ for $i \neq j$.   
  By inductive assumption there is $g \in T(n-1)$ such that $g[1]=b_1 \lor b_2$ and $g[i]=b_{i+1}$ for $2 \le i \le n$.
  Let $f(x_1, \ldots, x_n):=b_1(x_1, g(x_2, \ldots, x_n))$. 
  We prove that $f[i]=b_i$; we separate the case $i=1$, which does not depend on the inductive assumption,
  \begin{align*}
    f[1](x,y) & = b_1(x,g(y,\ldots,y)) & \\
    & = b_1(x,y) & g \text{ is idempotent} 
  \end{align*}
  from the case $i\neq1$:
  \begin{align*}
    f[i](x,y) & = b_1(y,g[i-1](x,y)) & \\
    & = \lnot b_1 (g[i-1](x,y),y) & \\
    & = (\lnot b_1 \land g[i-1])(x,y) &  \\
    & = (\lnot b_1 \land b_i)(x,y) & \text{by inductive assumption}
  \end{align*}
  for all $2 \le i \le n$, but $\lnot b_1 \land b_i =b_i$ as $b_1 \lor \cdots \lor b_n =1$ and $b_i \land  b_j =0$ for $i \neq j$. 
  This concludes the proof. 
\end{proof}

This leads us to a new proof of one of the results of \cite[Theorem 4.9]{G24}.

\begin{lemma}
  \label{lem:iso}
  Let $B$ be a Boolean ring. 
  The ring $H_B(2)$ of Lemma \ref{lem:hyper-bool} 
  is isomorphic to $B$. 
\end{lemma}
\begin{proof}
    Recall that $H_B(n)=\{(b_1, \ldots, b_n) \in B^n : b_1 \lor \cdots \lor b_n =1 \text{ and } b_i \land b_j =0 \text{ for } i \neq j\}$. 
    In particular $H_B(2) =\{(b,\lnot b) : b \in B\}$. 
    Then, by definition: 
    \begin{align*}
      (a, 1 -a) \lor (b,1- b) & = 
      \begin{bmatrix}
      1 & b \\
      0 & 1-b
      \end{bmatrix} 
      \begin{bmatrix}
      a \\
      1-a 
      \end{bmatrix} 
      = (a \lor b, 1-(a \lor b)) \\
      (a,1-b) \land (b,1-b) & = 
      \begin{bmatrix}
      b & 0 \\
      1-b & 1 
      \end{bmatrix} 
      \begin{bmatrix}
      a \\
      1-a 
      \end{bmatrix} = 
      (a \land b, 1-(a \land b)) \\
      \lnot (b,1-b) & = 
      \begin{bmatrix}
      0 & 1\\
      1 & 0
      \end{bmatrix} 
      \begin{bmatrix}
      b \\
      1-b
      \end{bmatrix} = 
      (\lnot b, b) 
    \end{align*}
    This concludes the proof. 
\end{proof}

\begin{theorem}
  \label{thm:equibool}
  The full subcategory of hyperaffine theories different from $U'$ is equivalent to the category of Boolean rings. 
\end{theorem}
\begin{proof}
    Let $H \neq U'$ be a hyperaffine theory, and let $B$ be a Boolean ring. 
    The two assignments $H \mapsto H(2)$ and $B \mapsto H_B$ are functorial. 
    By Lemma \ref{lem:iso} $B$ is isomorphic to $H_B(2)$. 
    For each $n \in \mathbb{N}$, we consider $\varphi: H(n) \to H(2)^n$ given by $\varphi(f) = (f[1], \ldots, f[n])$. 
    By Proposition \ref{prop:coord2} (1) $\varphi$ is well-defined, and by Proposition \ref{prop:coord2} (2) and (3) $\varphi$ is a bijection onto the image $\{(b_1, \ldots, b_n) \in H(2)^n : b_1 \lor \cdots \lor b_n =1, b_i b_j =0 \}$. 
    We prove that $\varphi$ is a morphism of theories, i.e. that for every $f \in H(n)$ and $g_1, \ldots, g_n \in H(m)$
  \begin{equation*}
    \varphi(f)(\varphi(g_1), \ldots, \varphi(g_n))(x_1, \ldots, x_m) = \varphi(f(g_1, \ldots,g_n))(x_1, \ldots, x_m) \text{.} 
  \end{equation*}
  By definition of $\varphi$ this amounts to prove that 
  \begin{equation*}
    \begin{bmatrix}
      g_1[1] & \cdots & g_n[1] \\
      \vdots & \ddots & \vdots \\
      g_1[m] & \cdots & g_n[m] 
    \end{bmatrix}
    \begin{bmatrix}
      f[1] \\
      \vdots \\
      f[n]
    \end{bmatrix}
    (x,y) = 
    \begin{bmatrix}
    f(g_1, \ldots, g_n)[1] \\
    \vdots \\
    f(g_1, \ldots, g_n)[m]
    \end{bmatrix}
    (x,y) \text{,}
  \end{equation*}
  i.e. that for every $i=1, \ldots, m$, 
  \begin{equation*}
    f[1]g_1[i] (x,y) \lor \cdots  \lor f[n]g_n[i](x,y)  = f(g_1[i](x,y), \ldots, g_n[i](x,y)) \text{.}
  \end{equation*} 
  For simplicity's sake, fix $i=1$. 
  Employing idempotence of $f$ and \eqref{eq:c5} Lemma \ref{lem:c5}, one can see that 
  \begin{align*}
    f[1]g_1[1] (x,y) \lor f[2]g_2[1](x,y) = f(g_1[1](x,y), g_2[1](x,y), y, \ldots, y)
  \end{align*}
  and then, inductively, 
  $$f[1]g_1[1] (x,y) \lor \cdots \lor f[n]g_n[1](x,y)=f(g_1[1](x,y), \ldots, g_n[1](x,y))\text{.}$$ 
  As the calculation with $i=2, \ldots, n$ is similar, we obtain the desired result. 
\end{proof}

\subsection{Commutative rings and affine theories}
In the case of affine theories, 
the assumption that every operation is a decomposition operation (cf. Definition \ref{def:icd}) is exchanged for 
the requirement of the existence of a Mal'cev operation. 
Note that the only model of a Mal'cev operation $p$ that splits is degenerate: 
\begin{equation*}
  y=p(y,x,x)=p(p(x,x,y), p(x,y,y), p(x,y,y)) = p(x,y,y) = x \text{,}
\end{equation*}
and therefore there are only two theories that are affine and hyperaffine at the same time: the degenerate ones. 

\begin{definition}
  \label{def:ar}
  For a given commutative ring $R$, we define the algebraic theory $A_R$ whose operations in $A_R(n)$ are $(r_1, \ldots, r_n) \in R^n$ such that $r_1 + \cdots +r_n =1$. 
\end{definition}


\begin{lemma}
  The algebraic theory $A_R$ is idempotent and commutative. 
  Moreover, the operation $(1,-1,1) \in A_R(3)$ is a Mal'cev operation. 
\end{lemma}
 \begin{proof}
    The proof is similar to that of Lemma \ref{lem:bool-hyper}. 
    The commutativity of the monoid $(R, \cdot, 1)$ ensures that $A_R$ is commutative, and the condition $r_1 + \cdots + r_n =1$ that $A_R$ is idempotent.
    Finally, $(1,-1,1)$ is a Mal'cev operation: 
    \begin{equation*} 
    \begin{bmatrix}
      1 \\
      -1\\
      1 
      \end{bmatrix} \circ  
      \begin{bmatrix}
      a_1 & b_1 & b_1 \\
      \vdots & \vdots & \vdots \\
      a_m & b_m & b_m
      \end{bmatrix} =
      \begin{bmatrix}
      a_1\\
      \vdots \\
      a_m 
      \end{bmatrix}
\end{equation*} 
for any $(a_1, \ldots, a_m), (b_1, \ldots, b_m) \in T(m)$ and similarly for the other condition. 
\end{proof}

We prove the analogous of Lemma \ref{lem:hyper-bool}. 
See \cite[Theorem 6.3.3]{RS02} for the similar result stating that a free Mal'cev mode on two generators can be endowed with the structure of commutative ring. 

\begin{lemma}
  \label{lem:ring}
    If $T$ is affine, then $T(2)$ is a commutative ring with the operations defined as follows, for every $a(x,y), b(x,y) \in T(2)$: 
  \begin{itemize}
    \item $(a + b)(x,y) := p(a(x,y),y,b(x,y))$; 
    \item $(a \cdot b)(x,y) := a(b(x,y),y)$; 
    \item $-a(x,y):=p(y,a(x,y),y)$; 
    \item $0(x,y):=y$ and $1(x,y):=x$ 
  \end{itemize} 
\end{lemma}
\begin{proof}
    Associativity and commutativity of the addition follow from Lemma \ref{lem:asscomm} (M1) and (M2). 
    Moreover, 
    \begin{align*}
      (a + 0)(x,y) & = p(a(x,y), y, y) & \\
      & = a(x,y) & p \text{ is Mal'cev} 
    \end{align*}
    \begin{align*}
      (a + (-a))(x,y) & = p(a(x,y),y,p(y,a(x,y),y)) & \\
      & = p(p(a(x,y),y,y),a(x,y),y) & \text{by Lemma \ref{lem:asscomm} (M1)}  \\
      & = p(a(x,y), a(x,y), y) & p \text{ is Mal'cev} \\ 
      & = y & p \text{ is Mal'cev} \\
      & = 0(x,y) \text{.} 
    \end{align*}
    Associativity and unitality of the multiplication are immediate. 
    We prove commutativity: 
    \begin{align*}
      (a \cdot b)(x,y) & = a(b(x,y), y) & \\
      & = a(b(x,y), b(y,y)) & b \text{ is idempotent}  \\ 
      & = b(a(x,y), a(y,y)) & a,b \text{ commute} \\
      & = b(a(x,y), y) & a \text{ is idempotent} \\
      & = (b\cdot a)(x,y)  \text{.} 
    \end{align*}
    Finally, 
    \begin{align*}
      r(a+b)(x,y) & = r(p(a(x,y),y,b(x,y)),y) & \\
      & = r(p(a(x,y),y,b(x,y)),p(y,y,y)) & p \text{ is idempotent}  \\ 
      & = p(r(a(x,y),y), r(y,y), r(b(x,y),y)) & p,r \text{ commute}  \\ 
      & = p(r(a(x,y),y), y, r(b(x,y),y)) & r \text{ is idempotent} \\
      & = (ra +rb)(x,y) \text{.} 
    \end{align*}
    This completes the proof. 
  \end{proof}

\begin{lemma}
  \label{lem:pl}
  If $T$ is an affine theory, then 
  \begin{equation*}
    p(a(x,y), b(x,y), c(x,y))=(a-b+c)(x,y)\text{.}
  \end{equation*}
\end{lemma}
\begin{proof}
  Let $p$ be the Mal'cev operation. Then, exploiting Lemma \ref{lem:asscomm}
  \begin{align*}
    p(a(x,y),b(x,y),c(x,y)) & = p(a(x,y),b(x,y),p(y,y,c(x,y))) & \\
    & = p(p(a(x,y),b(x,y),y),y,c(x,y)) & \text{(M1)} \\
    & = p(p(p(a(x,y),y,y),b(x,y),y),y,c(x,y)) & \\
    & = p(p(a(x,y),y,p(y,b(x,y),y)),y,c(x,y)) & \text{(M1)}\\
    & = p((a-b)(x,y),y,c(x,y)) & \\
    & = (a-b+c)(x,y) & 
  \end{align*}
  for every $a,b,c \in T(2)$. 
\end{proof}


\begin{lemma}
  \label{lem:unit}
  Let $R$ be a commutative ring. 
  The ring $A_R(2)$ of Lemma \ref{lem:ring} 
  is isomorphic to $R$. 
\end{lemma}
\begin{proof}
    Recall that $A_R(n)=\{(r_1, \ldots, r_n) \in R^n : r_1 + \cdots +r_n =1\}$. 
    In particular $A_R(2) =\{(r,1-r) : r \in R\}$. 
    Then, by definition: 
    \begin{align*}
      (r,1-r) + (s,1-s) & = 
      \begin{bmatrix}
      r & 0 & s\\
      1-r & 1 & 1-s
      \end{bmatrix} 
      \begin{bmatrix}
      1 \\
      -1 \\
      1
      \end{bmatrix} = 
      (r+s, 1-r-s) \\
      (r,1-r) \cdot (s,1-s) & = 
      \begin{bmatrix}
      s & 0 \\
      1-s & 1 
      \end{bmatrix} 
      \begin{bmatrix}
      r \\
      1-r 
      \end{bmatrix} = 
      (rs, 1-rs) \\
      -(r,1-r) & = 
      \begin{bmatrix}
      0 & r & 0\\
      1 & 1-r & 1
      \end{bmatrix} 
      \begin{bmatrix}
      1 \\
      -1 \\
      1
      \end{bmatrix} = 
      (-r, 1+r) 
    \end{align*}
    This concludes the proof. 
\end{proof}

Now, we prove the analogue of Proposition \ref{prop:coord2}.

\begin{proposition}
  \label{prop:coord}
  Let $T$ be an affine theory.
  Let $R$ be the commutative ring $T(2)$. 
  The following hold: 
  \begin{enumerate}
    \item if $f \in T(n)$, then $f[1]+\cdots+f[n]=1$ in $R$; 
    \item for $f,g \in T(n)$, if $f[i]=g[i]$ for all $1 \le i \le n$, then $f=g$; 
    \item for every $(r_1, \ldots, r_n) \in R^n$ such that $r_1+\cdots+r_n=1$, there is $f \in T(n)$ such that $f[i]=r_i$ for all $1 \le i \le n$. 
  \end{enumerate}
\end{proposition}
\begin{proof}
    Regarding the first item, we prove that $f[1] + \cdots + f[n] = 1$. 
    Observe that 
    \begin{align*}
      (f[1]+f[2])(x,y) & = p(f[1](x,y),y,f[2](x,y)) & \\
      & = p(f(x,y, \ldots, y) , f(y,y \ldots, y), f(y,x, \ldots, y)) & f \text{ is idemp.}\\
      & = f(p(x,y,y), p(y,y,x), \ldots, p(y,y,y)) & p,f \text{ commute} \\
      & = f(x,x,y \ldots, y) & p \text{ is Mal'cev} 
    \end{align*}
    and, therefore, inductively, 
    \begin{equation*}
      (f[1]+ \cdots + f[n])(x,y) = f(x,x, \ldots,x) =x \text{.}
    \end{equation*}

    For the second item, observe that 
    \begin{align*}
      p(f[1](x_1,y),y,f[2](x_2,y)) & = p(f(x_1,y, \ldots, y) , f(y,y \ldots, y), f(y,x_2, \ldots, y)) & \\
      & = f(p(x_1,y,y), p(y,y,x_2), \ldots, p(y,y,y)) &  \\
      & = f(x_1,x_2,y \ldots, y) &
    \end{align*}
    and, therefore, inductively, 
    \begin{equation*}
      p(f(x_1,\ldots,x_{n-1},y), y,f[n](x_n,y)) = f(x_1,x_2, \ldots,x_n)\text{.}
    \end{equation*}
    This implies that if $f[i]=g[i]$ for all $i=1, \ldots, n$, then $f = g$. 

  Finally, we prove the third item. 
  The proof is by induction on $n$. 
  If $n=1$, then clearly $f(x)=x$. 
  Assume that the result holds for $n-1$, and let $(r_1, \ldots, r_n) \in R^n$ be such that $r_1 + \cdots +r_n=1$. 
  By inductive assumption there is $g \in T(n-1)$ such that $g[1]=r_1 + r_2$ and $g[i]=r_{i+1}$ for $2 \le i \le n$.
  Let $f(x_1, \ldots, x_n):=p(r_2(x_2,x_1),x_1,g(x_1,x_3,\ldots,x_n))$. 
  We prove that $f[i]=r_i$; we separate three cases $i=1,2$ and $3 \le i \le n$:  
  \begin{align*}
    f(x,y,\ldots,y) & = p(r_2(y,x),x,g(x,y,\ldots,y)) & \\
    & = r_2(y,x) -x +g[1](x,y) & \text{ by Lemma \ref{lem:pl}}\\
    & = (1-r_2 -1+r_1+r_2)(x,y) & \text{ by inductive assumption} \\
    & = r_1(x,y)
  \end{align*} 
  \begin{align*} 
    f(y,x,y,\ldots,y) & = p(r_2(x,y),y,g(y, \ldots, y)) & \\
    & = p(r_2(x,y),y,y) & g \text{ is idempotent}\\
    & = r_2(x,y) & p \text{ is Mal'cev} 
  \end{align*}
  \begin{align*}
    f[i](x,y) & = p(r_2(y,y),y,g[i-1](x,y)) & 3 \le i \le n \\
    & = p(y,y,g[i-1](x,y)) & r_2 \text{ is idempotent} \\
    & = g[i-1](x,y) & p \text{ is Mal'cev} \\
    & = r_i(x,y) & \text{ by inductive assumption} 
  \end{align*}
  This concludes the proof. 
\end{proof}

We have all the ingredients to present one of the main theorems. 

\begin{theorem}
  \label{thm:equiring}
  The full subcategory of affine theories different from $U'$ is equivalent to the category of commutative rings. 
\end{theorem}
\begin{proof}
  Let $A\neq U'$ be an affine theory, and let $R$ be  a commutative ring. 
  The two assignments $A \mapsto A(2)$ and $R \mapsto A_R$ are functorial. 
  Lemma \ref{lem:unit} says that $R$ is isomorphic to $A_R(2)$. 
  For each $n \in \mathbb{N}$, consider the function $\varphi: A(n) \to A(2)^n$ such that $\varphi(f) = (f[1], \ldots, f[n])$. 
  By Proposition \ref{prop:coord}, $\varphi$ is a bijection onto the image $\{(r_1, \ldots, r_n) \in A(2)^n : r_1 + \cdots + r_n =1\}$. 
  We prove that $\varphi$ is a morphism of theories, i.e. that for every $f \in A(n)$ and $g_1, \ldots, g_n \in A(m)$
  \begin{equation*}
    \varphi(f)(\varphi(g_1), \ldots, \varphi(g_n))(x_1, \ldots, x_m) = \varphi(f(g_1, \ldots,g_n))(x_1, \ldots, x_m) \text{.} 
  \end{equation*}
  By definition of $\varphi$ this amounts to prove that 
  \begin{equation*}
    \begin{bmatrix}
      g_1[1] & \cdots & g_n[1] \\
      \vdots & \ddots & \vdots \\
      g_1[m] & \cdots & g_n[m] 
    \end{bmatrix}
    \begin{bmatrix}
      f[1] \\
      \vdots \\
      f[n]
    \end{bmatrix}
    (x,y) = 
    \begin{bmatrix}
    f(g_1, \ldots, g_n)[1] \\
    \vdots \\
    f(g_1, \ldots, g_n)[m]
    \end{bmatrix}
    (x,y) \text{,}
  \end{equation*}
  i.e. that for every $i=1, \ldots, m$, 
  \begin{equation*}
    f[1]g_1[i] (x,y) + \cdots +f[n]g_n[i](x,y)=f(g_1[i](x,y), \ldots, g_n[i](x,y)) \text{.}
  \end{equation*} 
  For simplicity's sake, fix $i=1$. 
  Observe that 
  \begin{align*}
    f[1]g_1[1] (x,y) + f[2]g_2[1](x,y) & = p(f[1](g_1[1](x,y),y),y, f[2](g_2[1](x,y),y)) \\
    & = p(f[1](g_1[1](x,y),y),f(y,\ldots,y), f[2](g_2[1](x,y),y))  \\
    & = f(p(g_1[1](x,y),y,y),p(y,y,g_2[1](x,y)), y, \ldots, y) \\
    & = f(g_1[1](x,y), g_2[1](x,y), y, \ldots, y)
  \end{align*}
  and inductively, $f[1]g_1[1] (x,y) + \cdots +f[n]g_n[1](x,y)=f(g_1[1](x,y), \ldots, g_n[1](x,y))$ as desired.
  As the calculation with $i=2, \ldots, n$ is similar, we obtain the desired result. 
\end{proof}

\section{Hyperaffine theories and algebraic logic.}
\label{sec:nba}
  In \cite{Dicker} it is shown that the variety of Boolean algebras can be axiomatised using two nullary operations, $0$ and $1$, and a ternary operation $q$ called `conditional disjunction'. 
Expressed in terms of the Boolean operators, conditional disjunction is $q(a,x,y)=(a \land x) \lor (\lnot a \land y)$.
  The restriction of the conditions of Definition \ref{def:hyper-affine} defining hyperaffine theories to binary operations yields the following axiomatisation of Boolean algebras:
\begin{itemize}
  \item $q(a,1,0)=a$; \hfill 
  \item $q(1,a,b)=a$ and $q(0,a,b)=b$; 
  \item $q(a, q(b,x,y), q(c,x,y))= q(q(a,b,c),x,y)$; 
  \item $q(a,x,x)=x$; 
  \item $q(a, q(b,x,y), q(b,z,w)) = q(b, q(a,x,z), q(a,y,w))$; 
  \item $q(a, q(a,x,y), q(a,z,w)) = q(a,x,w)$. 
\end{itemize}
In \cite{BS,SBLP20,BLPS18} these axioms have been adapted from dimension $2$ to
dimension $n$, with $n$ constants $e_1, \ldots, e_{n}$ acting as projections and a generalised $(n+1)$-ary if-then-else $q$. 

\begin{definition} \cite[Definition 3.7]{BS}
  A \emph{Boolean algebra} of \emph{dimension $n$} ($n$BA) is a set $A$ endowed with nullary operations $e_1, \ldots, e_n$ and a $(n+1)$-ary operation $q : A^{n+1} \to A$ satisfying the following identities:
  \begin{description}
  \item[H1] $q(e_i,x_1,\ldots,x_n) = x_i$ for every $i = 1,\ldots,n$;
  \item[H2] $q(y,x,\ldots,x) = x$; 
  \item[H3] $q(y,q(y,x^1_1,x^1_2,...,x^1_n),\ldots,q(y,x^n_1,x^n_2,\ldots,x^n_n)) = q(y,x^1_1,\ldots,x^n_n)$; 
  \item[H4] $q(y,q(z_1,x^1_1, \ldots,x^1_n),\ldots,q(z_n,x^n_1, \ldots,x^n_n)) \\
  = q(q(y,z_1, \ldots, z_n),q(y, x^1_1,\ldots,x^n_1),\ldots,q(y,x^1_n,\ldots,x^n_n))$;
  \item[H5] $q(y,e_1,...,e_n) = y$.
\end{description}
\end{definition}

If $A$ is an $n$BA, let $t(a,a_1, a_2) := q(a,a_1, a_2, e_3, \ldots, e_n)$ for every $a,a_1,a_2 \in A$. 
The set $B_A:=\{a \in A : t(a,b,b)=b\}$ is a Boolean algebra with respect to the ternary operation $t$ and constants $e_1, e_2$; 
see \cite[Lemma 4]{BLPS18} and subsequent discussion. 
The Boolean algebra $B_A$ is called the \emph{Boolean algebra of the coordinates} of $A$. 
The \emph{coordinates} of an element $a \in A$ are the elements of $B_A$ defined as 
\begin{equation*}
  a[1]=q(a,e_1,e_2, \ldots, e_2), \ldots, a[n]=q(a,e_2, \ldots,e_2,e_1)\text{.}
\end{equation*} 
See \cite[Definition 10, Lemma 6]{BLPS18}.   
This is similar to the Definition \ref{def:coord} of coefficients given above. 

The connection with hyperaffine theories is the following.
\begin{theorem}
  \label{thm:nba}
  If $T$ is a hyperaffine theory, then $T(n)$ is an $n$BA by letting 
  $e_i:=\pi^n_i$ for $i =1,\ldots, n$ and by defining 
  \begin{equation*}
    q: T(n)^{n+1} \to T(n) \quad \text{ as } \quad \circ: T(n) \times T(n)^n \to T(n)\text{.}
  \end{equation*}
Conversely, for any sequence $(A_n, q_n, e^n_1, \ldots, e^n_n)_{n \ge 2}$ of $n$BAs such that $B_{A_n}=B_{A_m}$ for every $n,m$
there is, up to isomorphism of theories, a unique hyperaffine theory $T$ such that $T(n)=A_n$,  
  \begin{equation*}
    q_n(a, b^1, \ldots, b^n) = a \circ (b^1, \ldots, b^n) \quad \text{ and } \quad e^n_i = \pi^n_i
  \end{equation*}
  for every $a, b^1, \ldots, b^n \in A_n$ and every $i=1, \ldots, n$. 
\end{theorem}
\begin{proof}
  Let $T$ be hyperaffine and let $n \ge 1$.
  Then $T(n)$ satisfies (H1) by the definition of algebraic theory, more precisely condition \eqref{eq:c12}.  
  Identity (H2) follows from the fact that $T$ is idempotent, and (H3) because every operation splits. 
  Finally, $T(n)$ satisfies (H4) by Lemma \ref{lem:c5} and (H5) again by \eqref{eq:c12}. 

  Conversely, let $(A_n)_{n \ge 2}$ be a sequence of $n$BAs having the same Boolean algebra $B$ of the coordinates. 
  To avoid trivialities, we assume that $B$ is nondegenerate. 
  Then we let $T$ to be the hyperaffine theory $H_B$. 
  Consider the function $\psi: A_n \to H_B(n)$ mapping an element $a$ to its coordinates $\psi(a)=(a[1], \ldots, a[n])$. 
  By \cite[Lemma 8 (i)]{BLPS18} $\psi$ is well-defined. Moreover, it is injective by \cite[Lemma 9]{BLPS18} and a surjective homomorphism by \cite[Theorem 8]{BLPS18}. 
  The fact that $\pi^n_i = e^n_i$ for every $i=1, \ldots, n$ is immediate and that $q_n(a, b^1, \ldots, b^n) = a \circ (b^1, \ldots, b^n)$
  for every $a, b^1, \ldots, b^n \in A_n$ follows from \cite[Lemma 8(ii)]{BLPS18}.
\end{proof}

On the one hand, $n$BAs have inspired work aimed at generalising the proof theory of classical propositional logic to the case of an arbitrary finite number of symmetric truth values \cite{BCS24}.
On the other, they have been discovered to have a connection with the skew Boolean algebras introduced by Leech \cite{Leech1,Leech2,Leech3}. 
More precisely, it was shown in \cite{BS} that 
any $n$BA $A$ contains a skew cluster of isomorphic right-handed skew Boolean algebras, called its \emph{skew reducts};
moreover, the skew reducts of $A$ allow to recover $A$ perfectly. 
By exploiting the correspondence between affine theories and commutative rings from Theorem \ref{thm:equiring}, we aim to define an $n$-dimensional commutative ring via 
axioms similar to (H1)-(H5). 
It is reasonable to expect that this approach may yield a definition of a skew ring and a result similar to the one established in \cite{BS} and mentioned previously. 
We defer this investigation to future research.

\section{Boolean actions} 
\label{sec:spaces}
Given a Boolean ring $B$, there are two ways to encode $B$ as an algebraic theory: 
either as the hyperaffine theory $H_B$ or as the affine theory $A_B$. 
Models of $H_B$ are the well-known $B$-sets \cite{B91}. 
It is therefore of interest to investigate models of an affine theory when the coefficient ring is Boolean, a task we undertake in this section. 
We observe that the models of $H_B$ highlight the set-theoretic properties of $B$, 
whereas the models of $A_B$ emphasise its algebraic structure 
(cf. Theorem 5.7 and the role of Boolean vector spaces).
We begin with the following observation. 

\begin{lemma}
  \label{prop:nec}
  Let $B$ be a Boolean ring. 
  If $X$ is a model of $A_B$ or a model of $H_B$, then $X$ is endowed with a binary action of $B$, i.e. a function $B\times X^2 \to X$, satisfying the following axioms:
\begin{description}
  \item[R1] $b(x,x)=x$; 
  \item[R2] $b(c(x,y), c(z,w)) = c(b(x,z),b(y,w))$; 
  \item[R3] $1(x,y) =x$ and $0(x,y)=y$; 
  \item[R4] $a(b(x,y),c(x,y)) = (ab+(1-a)c)(x,y)$   
\end{description}
for all $x,y,z,w \in X$ and $a,b,c \in R$. 
\end{lemma}
\begin{proof}
  By Lemma \ref{lem:unit} and Lemma \ref{lem:iso}, $A_B(2)$ and $H_B(2)$ are Boolean rings isomorphic to $B$. 
  Note that (R1) follows from the idempotence of $A_B$ (or $H_B$), (R2) from the commutativity, while (R3) and (R4) from the fact that $X$ is a model (cf. conditions \eqref{eq:alpha1} and \eqref{eq:alpha2} of Section 2). 
\end{proof}
 

A set of axioms similar to (R1)-(R4) was given by Bergman with the purpose of capturing the concept of action of $B$ on a set.

\begin{definition} \cite[Definition 7]{B91}
  \label{def:bset}
  Let $B$ be a Boolean ring. 
  A \emph{$B$-set} is a set $X$ together with a function $b: X^2 \to X$ for each $b \in B$ satisfying the following requirements: 
\begin{description}
  \item[B1] $b(x,x)=x$; 
  \item[B2] $b(b(x,y), z) = b(x,z)= b(x,b(y,z))$; 
  \item[B3] $(1-b)(x,y)=b(y,x)$;    
  \item[B4] $0(x,y) =y$;
  \item[B5] $b(c(x,y),y)=(bc)(x,y)$  
\end{description}
for all $x,y \in X$ and $b,c \in B$. 
\end{definition}

\begin{remark}
  \label{rem:balgebra}
  More generally, an action of $B$ on an algebra is an action on the underlying set of the algebra, such that each operation $b \in B$ commutes with each operation of the algebra.
\end{remark}

We establish that (B1)-(B5) can be rewritten as (R1)-(R4) plus an additional axiom. 

\begin{lemma}
  \label{lem:rtob}
    For a Boolean ring $B$, \emph{(B1)-(B5)} are equivalent to \emph{(R1)-(R4)} plus
    \begin{description}
        \item[R5] $b(b(x,y), b(z,w)) = b(x,w)$
    \end{description} 
    for all $x,y,z,w \in X$ and $b\in B$. 
\end{lemma}
    \begin{proof}
        Assume (R1)-(R5). 
        (R5) combined with (R1) immediately gives (B2).  
        (R4) and (R3) give (B3). 
        (B5) follows from (R4) with $c=0$.   
        Conversely, assume (B1)-(B5). 
        The content of \cite[Proposition 11]{B91} is precisely that (B1)-(B5) imply (R2). 
        (R3) is obvious in light of (B3) and (B4). 
        It is not difficult to see that (R5) follows repeatedly applying (B2). 
        To obtain (R4), first observe that $ab+(1-a)c = ab \lor (1-a)c$ as $ab(1-a)c=0$. 
        Moreover, it easy to see that 
        \begin{equation}
          \label{eq:or}
          (a \lor b)(x,y) = a(x,b(x,y)) \text{.}
        \end{equation}
        Therefore  
        \begin{align*}
          (ab \lor (1-a)c)(x,y) & = a(b(x,a(y,c(x,y))),a(y,c(x,y))) & \eqref{eq:or}, \text{(B3), (B5)} \\
          & = a(a(b(x,y),b(x,c(x,y))),a(y,c(x,y))) & \text{ (R2)} \\
          & = a(b(x,y), c(x,y)) & \text{ (R5)}
        \end{align*}
        This concludes the proof. 
\end{proof}

Algebras for a hyperaffine theory $T$ are exactly sets equipped with the Boolean action of $T(2)$.

\begin{corollary}
  \label{thm:models}
  Let $T$ be a hyperaffine theory, 
  and let $B:=T(2)$ be the Boolean ring associated with $T$. 
  Then the models of $T$ are sets equipped with an action of $B$ satisfying \emph{(R1)-(R5)}. 
\end{corollary}
\begin{proof}
  By \cite[Proposition 3.16]{G24}, and by Lemma \ref{lem:rtob}. 
\end{proof}

Unlike the case of hyperaffine theories, binary operations do not \emph{completely} determine an affine theory, due to the presence of the ternary Mal'cev operation $p$. 
However, axioms (R1)-(R4) can describe models of an affine theory if the Mal'cev operation can be expressed as composition of binary operations, 
and this holds 
iff there are no ring homomorphisms from $R$ to $\mathbb{F}_2$, where $\mathbb{F}_2$ is the field with two elements (see e.g. \cite{S77,IKS78}).  
In this case, models of an affine theory $T$ are sets with a binary action of the ring $R:=T(2)$ satisfying (R1)-(R4). 

Axioms in the general case can be deduced as follows. 

\begin{proposition} 
  \label{prop:6.3.4}
  For an affine theory $T$, and associated ring $R:=T(2)$, models of $T$ are sets equipped with binary operations $a: X^2 \to X$ for each $a \in R$ satisfying (R1)-(R4) and a ternary operation $p: X^3 \to X$ satisfying 
\begin{description}
  \item[A1] $p$ is a Mal'cev operation;  
  \item[A2] $p$ commutes with itself and with every $a \in R$; 
  \item[A3] $p(y,x,y)=(1+1)(y,x)$; 
  \item[A4] $p(a(x,y), b(x,y), c(x,y)) = (a-b+c)(x,y)$.    
\end{description}
\end{proposition}
\begin{proof}
  A model of $T$ is an affine $R$-module, that is a $R$-module whose linear combinations $r_1x_1 + \cdots + r_n x_n$ satisfy $r_1 + \cdots + r_n =1$.
  Then the result follows from \cite[Theorem 6.3.4]{RS02}. 
\end{proof}

When $B$ is a Boolean ring, we obtain the following characterisation. 
\begin{theorem}
  \label{thm:new1}
  Let $B$ be a Boolean ring, and consider the affine theory $A_B$. 
  Models of $A_B$, together with the choice of an element $o \in X$, are precisely vector spaces over $\mathbb{F}_2$ equipped with an action of $B$ satisfying \emph{(R1)-(R4)} and such that 
  \begin{description}
  \item[L1] $b(x,y) + b(z,w) = b(x+z,y+w)$
  \end{description} 
  for every $a,b \in B$ and $x,y,z,w \in X$.
\end{theorem}
\begin{proof}
  If $X$ models $A_B$, then $X$ is endowed with an action of $B$ and has a ternary operation satisfying (R1)-(R4) and (A1)-(A4).  
  We endow $X$ with the structure of $\mathbb{F}_2$-vector space as follows.
  We define $x+y:=p(x,o,y)$; 
  this sum is associative and commutative by (A1), (A2) and Lemma \ref{lem:asscomm}. 
  The element $o \in X$ is a zero for the sum by (A1). 
  Moreover, the inverse of $x$ is given by $x$ itself as (A3) shows: 
  \begin{equation*}
    x+x = p(x,o,x) = 0(x,o) =o\text{.}
  \end{equation*}
  Then, we add the action of $\mathbb{F}_2$: $1x:=x$, $0x:=o$. 
  The only nontrivial axiom is proven using (A3): 
  \begin{equation*}
    (1+1)x=0x=o=0(x,o)=(1+1)(x,o)=p(x,o,x)=x+x=1x+1x \text{.}
  \end{equation*}
  Finally, using (A2) it is easy to derive (L1). 
  Conversely, given an $\mathbb{F}_2$-vector space $(X,+,0)$ with an action of $B$ that satisfies (R1)-(R4) and (L1), we define $p(x,y,z)=x+y+z$. 
  As $X$ has characteristic $2$, (A1) is verified. 
  (A3) follows by definition of $p$: 
  \begin{equation*}
    p(y,x,y)=y+x+y=(1+1)y+x=x=0(y,x)=(1+1)(y,x) \text{.}
  \end{equation*}
  Moreover, (A2) follows from (L1). 
  Now, observe that 
  \begin{align*}
    (a+b)(x,y) & = a(b(y,x),b(x,y)) & \text{(R4),(B3)} \\
    & = a(b(x,y)+b(y,x)+b(x,y),0+b(x,y)) & \\
    & = a(b(x,y)+b(y,x),0) + b(x,y) & \text{(L1),(R1)} \\
    & = a(b(x+y,y+x),0) + b(x,y) & \text{(L1)} \\
    & = a(b(x+y,x+y),0) + b(x,y) & \\
    & = a(x+y,0)+b(x,y) & \text{(L1),(R1)} \\
    & = a(x+y,y+y)+b(x,y) & \\
    & = a(x,y) + y + b(x,y) & \text{(L1),(R1)}
  \end{align*}
  and therefore 
  \begin{align*}
    (a+b+c)(x,y) & = a(b(c(x,y),c(y,x)),b(c(y,x),c(x,y))) & \text{(R4)} \\
    & = a(b(y,x)+x +c(y,x),b(x,y)+y+c(x,y)) & \\
    & = a(b(y,x),b(x,y)) +a(x,y) + a(c(y,x),c(x,y)) & \text{(L1)}\\
    & = a(x,y)+y+b(x,y) +a(x,y)+a(x,y)+y+c(x,y) & \\
    & = a(x,y) + b(x,y)+c(x,y) & 
  \end{align*}
  proving (A4). 
\end{proof}

We sum up the results of this section in Table \ref{tab:axioms}.
In the table $B$ is a Boolean ring and $R$ is a commutative ring. 

\begin{table}
  \caption{A summary of the results of this section}
    \label{tab:axioms}
\centering
\begin{tabular}{c c c}
\toprule
\textbf{Axioms} & \textbf{characterise} & \textbf{by virtue of}\\
\midrule
(B1)-(B5)  & models of $H_B$ & \cite[Proposition 3.16]{G24} \\
\midrule
(R1)-(R5)  & models of $H_B$ & Lemma \ref{lem:rtob} \\
\midrule 
(R1)-(R4), (A1)-(A4)  & models of $A_R$ & Proposition \ref{prop:6.3.4} \\
\midrule 
(R1)-(R4)  & models of $A_R$ with no $R \to \mathbb{F}_2$ & \cite{S77,IKS78} \\
\midrule 
(R1)-(R4), (L1) & models of $A_B$ & Theorem \ref{thm:new1} \\
\bottomrule
\end{tabular}
\end{table}

In \cite{M93}, the semantics of if-then-else is given in terms of an action of $B$ on a $\lor$-semilattice with $\bot$, in the sense of Remark~\ref{rem:balgebra}.
Theorem~\ref{thm:new1} suggests that it might be interesting to investigate a semantics of if-then-else as a weakened (omitting (R5)) linear action of $B$ on an $\mathbb{F}_2$-vector space. 

\subsection{Boolean actions and sheaves}
In this section we give a sheaf representation of the structures of Theorem \ref{thm:new1}.

We start with some notation and preliminaries.

Given a Boolean algebra $B$ we ambiguously denote by $B$ the category whose objects are the elements of $B$ and there is an arrow $b \to c$ if $b \le c$. 
Let $\EuScript{V}$ be a variety of algebras. 
A sheaf $S: B^{\op} \to \EuScript{V}$ is a functor 
such that $S(b) \simeq S(c) \times S(d)$ when $b=c \lor d$ and $cd=0$. 

We recall Bergman's result concerning $B$-sets that will be used in the proof of Theorem \ref{thm:new2}. 

\begin{theorem} \cite[Theorem 12]{B91}
  \label{thm:berg}
  Let $B$ be a Boolean ring and let $\EuScript{V}$ be a variety of algebras.  
  \begin{itemize}
    \item For any set $X \neq \varnothing$, to equip $X$ with the structure of $B$-set amounts to give a sheaf $S:B^{\op} \to \mathbf{Set}$ such that $X \simeq S(1)$. 
    \item For any $\varnothing \neq X \in \EuScript{V}$, to equip $X$ with an action of $B$ amounts to give a sheaf $S:B^{\op} \to \EuScript{V}$ such that $X \simeq S(1)$. 
  \end{itemize}
\end{theorem}

Thanks to Theorem \ref{thm:new1} we obtain the main result of this section, which is a representation for models of an affine theory whose ring of coefficients is Boolean. 
\begin{theorem}
  \label{thm:new2}
    Let $B$ be a Boolean ring, and let $A_B$ its corresponding affine theory. 
    Let $X$ be a nonempty model of $A_B$ together with the choice of an element $o \in X$. 
    To equip $X$ with an action of $B$ amounts to give a sheaf $S:B^{\op} \to \mathbf{Vect}_{\mathbb{F}_2}$ such that $X \simeq S(1)$.
\end{theorem}
\begin{proof}
  Firstly, observe that by Theorem \ref{thm:berg} and by Remark \ref{rem:balgebra} a sheaf $S:B^{\op} \to \mathbf{Vect}_{\mathbb{F}_2}$ can be equivalently described as a $\mathbb{F}_2$-vector space $X$ such that: 
  \begin{itemize}
    \item $X$ is a $B$-set, i.e. there is an action of $B$ on $X$ satisfying (B1)-(B5); 
    \item $X$ satisfies (L1), i.e. $b(x+y,z+w) = b(x,z)+b(y,w)$ for every $b \in B$ and $x,y,z,w \in X$. 
  \end{itemize}
  Given a model $X$ of $A_B$ together with an element $o \in X$, we define an $\mathbb{F}_2$-vector space as in Theorem \ref{thm:new1}. 
  Now, $X$ has an action of $B$ that satisfies (R1)-(R4) and (L1) by Theorem \ref{thm:new1}.  
  By Lemma \ref{lem:rtob}, it is enough to prove that (R5) holds. 
  We prove, equivalently, that (B2) holds (here we prove the first part of (B2), the other part being similar).
  Using the usual algebraic manipulations involving the Mal'cev operation, we can see that:  
  \begin{equation}
    \label{eq:comb}
    b(x,y)=b(x,o)+(1-b)(y,o)\text{.}
  \end{equation}
  This implies that:
  \begin{align*}
    b(b(x,y),z) & = b(b(x,o)+(1-b)(y,o)) +(1-b)(z,o) & \eqref{eq:comb} \\
    & = b(b(x,o),o)+b((1-b)(y,o),o)+(1-b)(z,o) & \text{(L1)} \\
    & = b(x,o)+0(x,o)+(1-b)(z,o) & \text{(B5)} \\
    & = b(x,z) & \eqref{eq:comb}
  \end{align*} 
  as desiderd. 
  The converse is obvious in light of Theorem \ref{thm:new1}. 
\end{proof}

\section{Conclusions and vistas}
\label{sec:conclusions}
This work has shown that affine and hyperaffine algebraic theories form a natural setting for representing commutative and Boolean rings, respectively.  
By developing a framework that accommodates both constructions, we have clarified the features that make affine theories the counterpart of commutative rings, and hyperaffine theories the counterpart of Boolean rings. 
Moreover, the study of the models of both reveals a close connection with the algebraic treatment of the if-then-else construct.
Several directions arise.
A first step is to generalise the correspondence established here beyond rings, toward suitable classes of commutative \emph{semi}rings.
A second line of investigation concerns the algebras described in Theorem~\ref{thm:new1}, which suggest an equational foundation for an if-then-else operator in linear contexts, where the combination of programs obeys algebraic rather than set-theoretic structure.
Finally, a deeper understanding of the free algebras of affine theories
might lead 
to a notion of $n$-dimensional commutative ring and 
to an extension of the classical theory of affine combinations into higher dimensions.
%
%
%
\bibliographystyle{abbrv}
\bibliography{csl}

\begin{thebibliography}{10}

\bibitem{ARV10}
J.~Adámek, J.~Rosický, and E.~M. Vitale.
\newblock {\em Algebraic Theories: A Categorical Introduction to General
  Algebra}.
\newblock Cambridge Tracts in Mathematics. Cambridge U. Press, 2010.

\bibitem{B91}
G.~M. Bergman.
\newblock Actions of {B}oolean rings on sets.
\newblock {\em Algebra Univ.}, 28:153--187, 1991.

\bibitem{B2}
F.~Borceux.
\newblock {\em Handbook of Categorical Algebra. Volume 2: Categories and
  Structures}, volume~51 of {\em Encyclopedia of Mathematics and its
  Applications}.
\newblock Cambridge University Press, 1994.

\bibitem{BB04}
F.~Borceux and D.~Bourn.
\newblock {\em Mal'cev, Protomodular, Homological and Semi-Abelian Categories},
  volume 566 of {\em Mathematics and Its Applications}.
\newblock Springer-Dordrecht, 2004.

\bibitem{BCS24}
A.~Bucciarelli, P.-L. Curien, A.~Ledda, F.~Paoli, and A.~Salibra.
\newblock {The higher dimensional propositional calculus}.
\newblock {\em Logic Journal of the IGPL}, 100, 2024.

\bibitem{BLPS18}
A.~Bucciarelli, A.~Ledda, F.~Paoli, and A.~Salibra.
\newblock Boolean-like algebras of finite dimension: From {B}oolean products to
  semiring products.
\newblock In J.~Malinowski and R.~Palczewski, editors, {\em Janusz Czelakowski
  on Logical Consequence. Outstanding Contributions to Logic}, volume~27.
  Springer-Verlag, 2024.

\bibitem{BS}
A.~Bucciarelli and A.~Salibra.
\newblock On noncommutative generalisations of {B}oolean algebras.
\newblock {\em The Art of Discrete and Applied Mathematics}, 2(2), 2019.

\bibitem{C75}
B.~Csákány.
\newblock Varieties of modules and affine modules.
\newblock {\em Acta Mathematica Academiae Scientiarum Hungaricae}, 26:263--265,
  1975.

\bibitem{Dicker}
R.~M. Dicker.
\newblock A set of independent axioms for {B}oolean algebra.
\newblock {\em Proceedings of the London Mathematical Society}, 1:20--30, 1963.

\bibitem{G24}
R.~Garner.
\newblock Cartesian closed varieties {I}: the classification theorem.
\newblock {\em Algebra Univ.}, 85(38):1--37, 2024.

\bibitem{G25}
R.~Garner.
\newblock Cartesian closed varieties {II}: links to algebra and
  self-similarity.
\newblock {\em Proceedings of the Royal Society of Edinburgh: Section A
  Mathematics}, pages 1--45, 2025.

\bibitem{IKS78}
J.~R. Isbell, M.~I. Klun, and S.~H. Schanuel.
\newblock Affine part of algebraic theories {II}.
\newblock {\em Canadian Journal of Mathematics}, 30(2):231--237, 1978.

\bibitem{JS09}
M.~Jackson and T.~Stokes.
\newblock Semigroups with if-then-else and halting programs.
\newblock {\em International Journal of Algebra and Computation},
  19(07):937--961, 2009.

\bibitem{J90}
P.~T. Johnstone.
\newblock Collapsed toposes and cartesian closed varieties.
\newblock {\em Journal of Algebra}, 129:446--480, 1990.

\bibitem{Law63}
F.~W. Lawvere.
\newblock Functorial semantics of algebraic theories.
\newblock {\em Proceedings of the National Academy of Science}, 50(5):869--872,
  1963.

\bibitem{Leech1}
J.~Leech.
\newblock Skew {B}oolean algebras.
\newblock {\em Algebra Univ.}, 27:497--506, 1990.

\bibitem{Leech2}
J.~Leech.
\newblock Recent developments in the theory of skew lattices.
\newblock {\em Semigroup Forum}, 52:7--24, 1996.

\bibitem{Leech3}
J.~Leech and M.~Spinks.
\newblock Skew {B}oolean algebras derived from generalized {B}oolean algebras.
\newblock {\em Algebra Univ.}, 58:287--302, 2008.

\bibitem{M90}
E.~G. Manes.
\newblock A transformational characterization of \emph{if-then-else}.
\newblock {\em Theoretical Computer Science}, 71:413--417, 1990.

\bibitem{M92}
E.~G. Manes.
\newblock Equations for if-then-else.
\newblock In S.~Brookes, M.~Main, A.~Melton, M.~Mislove, and D.~Schmidt,
  editors, {\em Mathematical Foundations of Programming Semantics}, pages
  446--456. Springer, 1992.

\bibitem{M93}
E.~G. Manes.
\newblock Adas and the equational theory of if-then-else.
\newblock {\em Algebra Univ.}, 30:373--394, 1993.

\bibitem{PRS95}
A.~Pilitowska, A.~Romanowska, and J.~D.~H. Smith.
\newblock Affine spaces and algebras of subalgebras.
\newblock {\em Algebra Univ.}, 34:527--540, 1995.

\bibitem{RS02}
A.~B. Romanowska and J.~D.~H. Smith.
\newblock {\em Modes}.
\newblock World Scientific, 2002.

\bibitem{SBLP20}
A.~Salibra, A.~Bucciarelli, A.~Ledda, and F.~Paoli.
\newblock Classical logic with $n$ truth values as a symmetric many-valued
  logic.
\newblock {\em Foundations of Science}, 28:115--142, 2023.

\bibitem{SO66}
J.~Schmidt and F.~Ostermann.
\newblock Der baryzentrische {K}alkül als axiomatische {G}rundlage der affinen
  {G}eometrie.
\newblock {\em Journal für die reine und angewandte {M}athematik}, 224:44--57,
  1966.

\bibitem{Smith}
J.~D.~H. Smith.
\newblock {\em Mal'cev Varieties}.
\newblock Lecture Notes in Mathematics. Springer Berlin, 1976.

\bibitem{S98a}
T.~Stokes.
\newblock Radical classes of algebras with ${B}$-action.
\newblock {\em Algebra Univ.}, 40:73--85, 1998.

\bibitem{S98b}
T.~Stokes.
\newblock Sets with ${B}$-action and linear algebra.
\newblock {\em Algebra Univ.}, 39:31--43, 1998.

\bibitem{S77}
A.~Szendrei.
\newblock On the arity of affine modules.
\newblock {\em Colloquium Mathematicae}, 38(1):1--4, 1977.

\bibitem{T93}
W.~Taylor.
\newblock Abstract clone theory.
\newblock In I.~G. Rosenberg and G.~Sabidussi, editors, {\em Algebras and
  Orders}, pages 507--530. Kuwer Academic Publisher, 1993.

\end{thebibliography}

\end{document}